\documentclass[a4paper,leqno,draft]{article}
\usepackage{amssymb,amsmath,amsfonts,amsthm,
mathrsfs}
\newtheorem{theorem}{Theorem}[section]
\newtheorem{corollary}[theorem]{Corollary}
\newtheorem{lemma}[theorem]{Lemma}
\newtheorem{proposition}[theorem]{Proposition}
\newtheorem{definition}[theorem]{Definition}
\theoremstyle{definition}
\newtheorem{remark}[theorem]{Remark}

\newcommand{\wt}[1]{\widetilde{#1}}

\newcommand{\Cinfc}{\ensuremath{\mathcal{C}^\infty_{\text{c}}}}
\newcommand{\D}{\ensuremath{{\mathcal D}}}
\renewcommand{\S}{\mathscr{S}}

\newcommand{\mF}{\mathcal{F}}

\newcommand{\dslash}{d\hspace{-0.4em}{ }^-\hspace{-0.2em}}

\newcommand{\mb}[1]{\ensuremath{\mathbb{#1}}}
\newcommand{\N}{\mb{N}}

\newcommand{\R}{\mb{R}}
\newcommand{\C}{\mb{C}}
\newcommand{\Z}{\mb{Z}}

\newcommand{\lara}[1]{\langle #1 \rangle}

\newcommand{\supp}{\mathrm{supp}}

\newfont{\bigmath}{cmr12 at 13pt}

\newfont{\grecomath}{cmmi12 at 15pt}

\newcommand{\esp}{\mathrm{e}}

\newfont{\bl}{msbm10 scaled \magstep2}

\newcommand{\beq}{\begin{equation}}
\newcommand{\eeq}{\end{equation}}

\newcommand{\notmid}{\mid\kern-0.5em\not\kern0.5em}

\newcommand{\inner}[3][\empty]{\ifx#1\empty\left( #2|#3\right)\else#1( #2|#3 #1)\fi}

\newcommand{\eps}{\varepsilon}

\renewcommand{\Re}{\ensuremath{\mathrm{Re}}}

\newcommand{\mC}{\mathcal{C}}

\begin{document}

\title{{\bf $L^p$ and Sobolev boundedness of pseudodifferential operators with non-regular symbol: a regularisation approach}}
\author{Claudia Garetto\footnote{Supported by JRF (Imperial College London)} \\
Department of Mathematics\\
Imperial College London\\
\texttt{c.garetto@imperial.ac.uk}\\
}
\date{ }
\maketitle
\begin{abstract}
In this paper we investigate $L^p$ and Sobolev boundedness of a certain class of pseudodifferential operators with non-regular symbols. We employ regularisation methods, namely convolution with a net of mollifiers $(\rho_\eps)_\eps$, and we study the  corresponding net of pseudodifferential operators providing $L^p$ and Sobolev estimates which relate the parameter $\eps$ with the non-regularity of the symbol.
\end{abstract}


\setcounter{section}{-1}
\section{Introduction}
The mathematical research in the field of pseudodifferential operators most frequently concentrates on operators with smooth symbols. However, applications to several problems in PDE, from nonlinear problems to problems on non-smooth domains, require symbols with minimal smoothness, i.e., non-regular symbols. With this expression we mean symbols which are smooth in the variable $\xi$ but less regular in $x$ (for instance in some Sobolev space or H\"older-Zygmund class). Pseudodifferential operators with non-regular symbol have been studied by different authors in particular in connection with their mapping properties on $L^p$ and Sobolev spaces. We recall the work of M. Nagase and H. Kumano-go at the end of the 70's in \cite{KuNa:78, Nagase:77}, the work of J. Marschall for differential operators with Sobolev coefficients in \cite{Marschall:87, Marschall:88} and the deep study of pseudodifferential operators with non-regular symbol of H\"older-Zygmund type in \cite[Chapters 1, 2]{Taylor:91} and \cite[Chapter 1]{Taylor:00}. For some recent work on non-regular symbols and boundedness results on $L^p$ spaces, Besov spaces or weighted Sobolev spaces we refer the reader to \cite{GM:04, GM:05, GM:09, GM:09a}.

In this paper we study pseudodifferential operators with non-regular symbol in the exotic class $\mathcal{C}^r_\ast S^m_{1,1}$, i.e., of type $(1,1)$, smooth in $\xi$ and in the class $\mathcal{C}^r_\ast$ with respect to $x$. Inspired by the continuity result of G. Bourdaud in \cite{Bou:88} and the pioneering work of Stein \cite{Stein:72}, Michael Taylor has proved that if $r>0$ a pseudodifferential operator with symbol in $\mathcal{C}^0_\ast S^m_{1,1}$ is bounded on $H^{s,p}$ provided that $0<s<r$ and $p\in(1,\infty)$ (see \cite[Theorem 2.1.A]{Taylor:91}). Our aim is to enlarge the family of Sobolev spaces on which this type of non-regular pseudodifferential operators are bounded. We do this via approximation/regularisation methods, in the sense that given $a\in\mathcal{C}^r_\ast S^0_{1,1}$ we study a net of pseudodifferential operators $a_\eps(x,D)$ with regular symbol $a_\eps$ converging to $a$ when $\eps$ tends to $0$. $a_\eps$ is regular in $x$ since it is obtained via convolution with a mollifier $\rho$, i.e., $a_\eps(x,\xi)=(a(\cdot,\xi)\ast\rho_\eps)(x)$, with $\rho_\eps(x)=\eps^{-n}\rho(x/\eps)$. From the boundedness result of Taylor we clearly expect a blow-up in $\eps$ when $s>r$. Our main achievement is a precise estimate of this blow-up which can still detach the boundedness on Sobolev spaces when $s\in(0,r)$. In detail, for a pseudodifferential operator with symbol $a$ in $\mathcal{C}^r_\ast S^0_{1,1}$ we get a blow-up of type $\eps^{-h}$ in the $H^{s,p}$-norm when $r>0$ and $s\in(0,r+h)$, i.e., $\Vert a_\eps(x,D)f\Vert_{H^{s,p}}\le C\eps^{-h}\Vert f\Vert_{H^{s,p}}$ for all $f\in H^{s,p}$ and $\eps\in(0,1]$. More general nets of pseudodifferential operators, depending on the parameter $\eps$ but not necessarily defined via convolution with a mollifier $\rho_\eps$, are investigated in the final part of the paper under the point of view of $L^p$ and Sobolev continuity.

We now describe the contents of the paper in more detail.

Section 1 provides some background on H\"older-Zygmund classes and their regularisation. It is inspired by the investigation of H\"older-Zygmund regularity in the Colombeau framework initiated by G. H\"ormann in \cite{GH:04} and is motivated by geophysical problems (see \cite{HdH:01, HdH:01b, HdH:01c}) modeled through differential equations with H\"older-Zygmund coefficients. More precisely, we study how the net $u\ast\rho_\eps$ depends on the parameter $\eps$ when $u\in\mathcal{C}^s_\ast(\R^n)$ and we compare its $\mathcal{C}^{s+r}_\ast$- and $\mathcal{C}^s_\ast$-norms for arbitrary positive $r$. We complete this section with some interpolation and continuity results for nets of linear operators which will be employ in the Sobolev context in Section 2. Section 2 is the core of the paper where the $L^p$- and $H^{s,p}$-boundedness of regularised nets of pseudodifferential operators is proved. The results obtained for operators with symbol $a_\eps(x,\xi)=(a(\cdot,\xi)\ast\rho_\eps)(x)$ when  $a$ in $\mathcal{C}^r_\ast S^m_{1,1}$ follow the line of proof of Taylor in \cite[Chapters 1, 2]{Taylor:91} and \cite[Chapter 1]{Taylor:00} and make use of concepts as $S^1_0$-partition of unity, equivalent Sobolev norms and decomposition into elementary symbols. The novelty represented by the parameter $\eps$ is certainly crucial and requires, at every step of the proof, precise estimates which keep track of it. Our main result is not the continuity estimate per se but rather the understanding of how the continuity constant (now depending on the parameter $\eps$) is related to the H\"older-Zygmund properties of the non-regular symbol $a$. By means of its refined methods and boundedness results, Section 2 also provides a new way to look at nets of pseudodifferential operators of the type recently studied in \cite{GGO:03, GH:05} in the framework of Colombeau algebras. The paper ends with Section 3 where we consider arbitrary nets $(a_\eps)_\eps$ of regular symbols and we prove $L^p$ and Sobolev boundedness of the corresponding nets of pseudodifferential operators. 
\section{Regularisation in the space $\mC^s_\ast(\R^n)$ and some notions of interpolation}
This section provides the technical background necessary for the investigation of $L^p$ and Sobolev boundedness in Section \ref{sec_Taylor}. We begin by studying the regularisation, via convolution with a mollifier, of tempered distributions in H\"older-Zygmund classes and we then pass to consider nets of linear operators acting on interpolation couples of Banach spaces.
\subsection{Regularisation via convolution with a mollifier in the space $\mC^s_\ast(\R^n)$ }
Following, \cite[Section 8.5]{Hoermander:97} we introduce the H\"older-Zygmund classes via a continuous Littlewood-Paley decomposition: let $\varphi\in\D(\R^n)$ real valued and symmetric such that $\varphi(\xi)=0$ for $|\xi|>1$ and $\varphi(\xi)=1$ for $|\xi|\le 1/2$. Let $\psi=\frac{d}{dt}\varphi(\xi/t)|_{t=1}=-\xi\cdot\nabla\varphi(\xi)$. Then,
\[
1=\varphi(\xi)+\int_1^{+\infty}\psi\biggl(\frac{\xi}{t}\biggr)\frac{dt}{t},
\]
and the decomposition formula
\[
u=\varphi(D)u+\int_1^{+\infty}\psi\biggl(\frac{D}{t}\biggr)u\frac{dt}{t}
\]
holds in $\S'(\R^n)$.

Let $s\in\R$ we define the Zygmund space $\mC^s_\ast(\R^n)$ as the set of all $u\in\S'(\R^n)$ such that
\[
\Vert u\Vert_{\mC^s_\ast}=\Vert \varphi(D)u\Vert_\infty+\sup_{t>1}t^s\Vert\psi(D/t)u\Vert_\infty<\infty
\]
with $\varphi$ and $\psi$ as above. Up to equivalence of the norm this definition is independent of the choice of the Littlewood-Paley decomposition $(\varphi,\psi)$. We recall that $\mC^s_\ast(\R^n)$ coincides with the H\"older space $\mC^s(\R^n)$ when $s>0$ is not integer. For a survey on H\"older and Zygmund classes we refer the reader to \cite[Sections 8.5, 8.6]{Hoermander:97} and \cite{Triebel:II}.

We now study the convolution of $u\in\mC^s_\ast(\R^n)$ with a \emph{mollifier} $\rho$, i.e. a function $\rho$ in $\S(\R^n)$ with $\int\rho=1$. More precisely we will convolve $u$ with the delta-net $\rho_\eps(x):=\eps^{-n}\rho(x/\eps)$. The following preliminary lemma can be found in \cite[Section 2.6, Lemma 12]{Meyer:92} and has been adapted to the case of rapidly decreasing functions in \cite[Lemma 17]{GH:04}. We recall that for $r\in\N$, $\S_r(\R^n)$ denotes the space of all smooth functions $f$ such that $\sup_{x\in\R^n}(1+|x|)^{m}|\partial^\alpha f(x)|<\infty$ for all $m\in\N$ and all $\alpha\in\N^n$ with $|\alpha|\le r$.
\begin{lemma}
\label{lemma_gue}
\leavevmode
\begin{itemize}
\item[(i)] Let $s,r\in\N$ with $0\le s\le r$. If $f\in \S_r(\R^n)$ has moments up to order $s$ vanishing, i.e., $\int x^\gamma f(x)dx=0$ for $|\gamma|\le s$ then there exist functions $f_\alpha\in \S_r(\R^n)$ with $|\alpha|=s$ such that
\[
f=\sum_{|\alpha|=s}\partial^\alpha f_\alpha.
\]
and $\int f_\alpha(x)\, dx=0$ for $|\alpha|=s$.
\item[(ii)] If $f\in\S(\R^n)$ has all the moments vanishing then the representation above holds for all $s\in\N$ with $f_\alpha\in\S(\R^n)$.
\end{itemize}
\end{lemma}
\begin{proposition}
\label{prop_ast}
Let $\rho$ be a mollifier in $\S(\R^n)$ and let $s\in\R$. For all $r\ge0$ there exists $C>0$ such that
\[
\Vert u\ast\rho_\eps\Vert_{\mC^{r+s}_\ast}\le C\eps^{-r}\Vert u\Vert_{\mC^s_\ast}
\]
holds for all $u\in \mC^s_\ast(\R^n)$ and all $\eps\in(0,1]$.
\end{proposition}
\begin{proof}
By definition of $\mC^s_\ast(\R^n)$ we have
\[
\Vert u\Vert_{\mC^s_\ast}=\Vert \varphi(D)u\Vert_\infty+\sup_{t>1}t^s\Vert\psi(D/t)u\Vert_\infty
\]
with $\varphi$ and $\psi$ as at the beginning of this subsection. By applying the operator $\varphi(D)$ to $u_\eps:=u\ast\rho_\eps$ we have that
\beq
\label{est_varphi(D)}
\Vert \varphi(D)(u_\eps)\Vert_\infty\le \Vert\varphi(D)u\Vert_\infty \Vert\rho_\eps\Vert_1\le \Vert\varphi(D)u\Vert_\infty.
\eeq
and
\[
\sup_{t>1}t^s\Vert\psi(D/t)u_\eps\Vert_\infty\le \sup_{t>1}t^s\Vert\psi(D/t)u\Vert_\infty\Vert\rho_\eps\Vert_1
\]
This means that $\Vert u_\eps\Vert_{\mC^{s}_\ast}\le \Vert u\Vert_{\mC^s_\ast}$ and therefore the case $r=0$ is trivial. Let us consider $r>0$ integer and let us take $\wt{\psi}\in\Cinfc(\R^n)$ with $\wt{\psi}=0$ near $0$ and $\wt{\psi}=1$ on $\supp(\psi)$. It follows that $\mF^{-1}\wt{\psi}$ has all the moments vanishing and that $\psi(D/t)u\ast\rho_\eps=\psi(D/t)u\ast\wt{\psi}(D/t)\rho_\eps$. Lemma \ref{lemma_gue}$(ii)$ applied to $\mF^{-1}\wt{\psi}$ allows us to find functions $\wt{\psi}_\alpha$ with $|\alpha|=r$ such that
\[
\mF^{-1}\wt{\psi}=\sum_{|\alpha|=r}D^\alpha(\mF^{-1}\wt{\psi}_\alpha)=
\sum_{|\alpha|=r}\mF^{-1}(\xi^\alpha\wt{\psi}_\alpha).
\]
Combining basic properties of the Fourier transform with the convolution we have that $\wt{\psi}(D/t)\rho_\eps$ can be written as
\[
t^{-r}\eps^{-r}\sum_{|\alpha|=r}\wt{\psi}_\alpha(D/t)(D^\alpha\rho)_\eps.
\]
This yields the estimate
\[
\Vert \psi(D/t)u_\eps\Vert_\infty\le t^{-r}\eps^{-r}\sum_{|\alpha|=r}\Vert\psi(D/t)u\ast\wt{\psi}_\alpha(D/t)(D^\alpha \rho)_\eps\Vert_\infty\le t^{-r-s}\eps^{-r}\Vert u\Vert_{\mC^s_\ast}\sum_{|\alpha|=r}\Vert\wt{\psi}_\alpha(D/t)(D^\alpha \rho)_\eps\Vert_1.
\]
Arguing as in the proof of Theorem 1.1 in \cite{GH:04} we see that when $t\ge\eps^{-1}$ any $\Vert\wt{\psi}_\alpha(D/t)(D^\alpha \rho)_\eps\Vert_1$ can be estimated by a constant $C$ depending only on $\rho$, $\wt{\psi}$, $r$ and $n$. In other words,
\beq
\label{first_est}
\sup_{t\ge\eps^{-1}}t^{r+s}\Vert \psi(D/t)u_\eps\Vert_\infty\le C\Vert u\Vert_{\mC^s_\ast}\eps^{-r}.
\eeq
Since,
\beq
\label{second_est}
\sup_{1<t\le\eps^{-1}}t^{r+s}\Vert\psi(D/t)u_\eps\Vert_\infty\le \sup_{1<t\le\eps^{-1}}t^r\Vert u\Vert_{\mC^s_\ast}\le\Vert u\Vert_{\mC^s_\ast}\eps^{-r},
\eeq
combining \eqref{first_est} with \eqref{second_est} we conclude that
\beq
\label{est_psi}
\sup_{t>1}t^{r+s}\Vert \psi(D/t)u_\eps\Vert_\infty\le C\Vert u\Vert_{\mC^s_\ast}\eps^{-r}.
\eeq
The estimates \eqref{est_varphi(D)} and \eqref{est_psi} show that there exists some constant $C$ such that for all $u\in\mC^s_\ast(\R^n)$,
\[
\Vert u_\eps\Vert_{\mC^{s+r}_\ast}\le C\eps^{-r}\Vert u\Vert_{\mC^s_\ast}.
\]
If now $r>0$ is not integer we have that the estimate
\[
\Vert \psi(D/t)u_\eps\Vert_\infty\le t^{-r'-s}\eps^{-r'}\Vert u\Vert_{\mC^s_\ast}\sum_{|\alpha|=r'}\Vert\wt{\psi}_\alpha(D/t)(D^\alpha g)_\eps\Vert_1.
\]
is valid for some integer $r'\ge r$. Under the hypothesis of $t\ge\eps^{-1}$ this leads to
\[
\Vert \psi(D/t)u_\eps\Vert_\infty\le C\Vert u\Vert_{\mC^s_\ast}t^{-r-s}\eps^{-r}(t\eps)^{-r'+r}\le C\Vert u\Vert_{\mC^s_\ast}t^{-r-s}\eps^{-r}.
\]
Since \eqref{est_varphi(D)} and \eqref{second_est} hold for every $r>0$ the proof is complete.
\end{proof}
\begin{corollary}
\label{corol_ast}
If $u\in\mC^s_\ast(\R^n)$ and $s+r>0$ then there exists a constant $C$ depending only on $r$ such that
\[
\Vert u\ast\rho_\eps\Vert_\infty\le C\eps^{-r}\Vert u\Vert_{\mC^s_\ast(\R^n)}.
\]
\end{corollary}
\begin{proof}
This corollary is easily proved by combining Proposition \ref{prop_ast} with the embedding $\mC^t_\ast\subseteq L^\infty$ valid for $t>0$ (see the decomposition formula for tempered distributions at the beginning of this subsection or \cite[2.3.2, Remark 3]{Triebel:II}).
\end{proof}
\begin{remark}
\label{rem_gue}
Corollary \ref{corol_ast} yields the estimate obtained by H\"ormann in \cite{GH:04} for the net $\Vert\partial^\alpha(u\ast\rho_\eps)\Vert_\infty$ when $\alpha=0$ and $s\neq 0$ (see Definition 3 and Theorem 7 in \cite{GH:04}). Note that by assuming that the mollifier $\rho$ has vanishing moments $\int x^\alpha \rho(x)\, dx$ when $\alpha\neq 0$, H\"ormann has proved a more precise estimate of the norm $\Vert u\ast\rho_\eps\Vert_\infty$ when $s=0$, namely $\Vert u\ast\rho_\eps\Vert_\infty=O(\log(1/\eps))$ as $\eps\to 0$.
\end{remark}

\subsection{Nets of linear operators and interpolation couples}
We conclude this first section by considering a net of operators $(T_\eps)_{\eps\in(0,1]}$ acting on the complex interpolation of a couple $\{A_0,A_1\}$ of Banach spaces. The notions of this subsections will be employ in Section \ref{sec_Taylor} for proving results of Sobolev boundedness.

We begin by recalling that given $A_0$ and $A_1$ complex Banach spaces, the couple $\{A_0,A_1\}$ is called interpolation couple if there exists a linear complex Hausdorff space $A$ such that both $A_0$ and $A_1$ are linearly and continuously embedded in $A$. It follows that $A_0+A_1$ is a well defined subset of $A$. In addition $A_0+A_1$ is a quasi-Banach space with respect to the quasi-norm $\Vert a\Vert=\inf \Vert a_0\Vert_{A_0}+\Vert a_1\Vert_{A_1}$, where the infimum is taken over all the representations $a=a_0+a_1$ with $a_0\in A_0$ and $a_1\in A_1$. Referring to \cite[Section 1.6]{Triebel:II} we now define the set of functions $F[A]$.
\begin{definition}
\label{def_FA}
Let $\{A_0,A_1\}$ be an interpolation couple of Banach spaces, $A=A_0+A_1$ and $\sigma=\{z\in\C: 0<\Re\, z<1\}$. $F[A]$ denotes the collection of all function $f$ on $\overline{\sigma}$ with values in $A$ such that
\begin{itemize}
\item[(i)] $f$ is $A$-continuous on $\overline{\sigma}$ and $A$-analytic in $\sigma$ with $\sup_{z\in\overline{\sigma}}\Vert f(z)\Vert_A <\infty$,
\item[(ii)] $f(it)\in A_0$ and  $f(1+it)\in A_1$ for all $t\in\R$, the corresponding operators from $\R$ to $A_0$ and $A_1$, respectively, are continuous and
\[
\Vert f\Vert_{F[A]}=\sup_{t\in\R}(\Vert f(it)\Vert_{A_0}+\Vert f(1+it)\Vert_{A_1})<\infty.
\]
\end{itemize}
\end{definition}
Note that $F[A]$ is a Banach space for the topology induced by the norm above.
\begin{definition}
\label{def_interp}
Let $\{A_0,A_1\}$ be an interpolation couple of Banach spaces. Let $A=A_0+A_1$ and $\theta\in(0,1)$. The interpolation space $[A_0,A_1]_\theta$ is the set of all $a\in A$ such that there exists $f\in F[A]$ with $f(\theta)=a$.
\end{definition}
$[A_0,A_1]_\theta$ is a Banach space with respect to the norm $\Vert a\Vert_{[A_0,A_1]_\theta}=\inf_f \Vert f\Vert_{F[A]}$, where the infimum is taken over all $f\in F[A]$ with $f(\theta)=a$.
\begin{proposition}
\label{prop_interpolation}
Let $\{A_0,A_1\}$ and $\{B_0,B_1\}$ be two interpolation couples and let $(T_\eps)_{\eps\in(0,1]}$ be a family of linear operators from $A=A_0+A_1$ to $B=B_0+B_1$ such that $T_\eps:A_j\to B_j$ is continuous for $j=0,1$ and all $\eps\in(0,1]$, i.e., 
\[
\forall j=0,1\, \exists(\omega_{j,\eps})_\eps\in\R^{(0,1]}\, \forall a\in A_j\, \forall\eps\in(0,1]\quad \Vert T_\eps a\Vert_{B_j}\le \omega_{j,\eps}\Vert a\Vert_{A_j}.
\]
Hence, for all $\theta\in(0,1)$ and all $\eps\in(0,1]$, the operator $T_\eps$ maps $[A_0,A_1]_\theta$ into $[B_0,B_1]_\theta$ and the inequality
\[
\Vert T_\eps a\Vert_{[B_0,B_1]_\theta}\le \max(\omega_{0,\eps},\omega_{1,\eps})\Vert a\Vert_{[A_0,A_1]_\theta}
\]
holds for all $a\in[A_0,A_1]_\theta$ and all $\eps\in(0,1]$.
\end{proposition}
\begin{proof}
We begin by noting that $T_\eps$ is continuous from $A$ to $B$. Indeed, by working with any representation $a_0+a_1$ of $a$ we get
\[
\Vert T_\eps a\Vert_B\le \Vert T_\eps a_0\Vert_{B_0}+\Vert T_\eps a_1\Vert_{B_1}\le \max(\omega_{0,\eps},\omega_{1,\eps})(\Vert a_0\Vert_{A_0}+\Vert a_1\Vert_{A_1})\le \max(\omega_{0,\eps},\omega_{1,\eps})\Vert a\Vert_A.
\]
It is easy to see that if $f\in F[A]$ then $T_\eps\circ f\in F[B]$ for all $\eps$. By definition of the norm $\Vert\cdot\Vert_{[B_0,B_1]_\theta}$ we have that
\[
\Vert T_\eps a\Vert_{[B_0,B_1]_\theta}\le \sup_{t\in\R}\Vert g(it)\Vert_{B_0}+\Vert g(1+it)\Vert_{B_1}
\]
for all $g\in F[B]$ with $g(\theta)=T_\eps a$. It follows that for $f\in F[A]$ with $f(\theta)=a$ we can write
\[
\Vert T_\eps a\Vert_{[B_0,B_1]_\theta}\le \sup_{t\in\R}\Vert (T_\eps\circ f)(it)\Vert_{B_0}+\Vert (T_\eps\circ f)(1+it)\Vert_{B_1}.
\]
The continuity of the operator $T_\eps$ restricted to $A_0$ and $A_1$ yields
\beq
\label{inter_ineq}
\Vert T_\eps a\Vert_{[B_0,B_1]_\theta}\le \sup_{t\in\R}(\omega_{0,\eps}\Vert f(it)\Vert_{A_0}+\omega_{1,\eps}\Vert f(1+it)\Vert_{A_1})\le \max(\omega_{0,\eps},\omega_{1,\eps})\Vert f\Vert_{F[A]}.
\eeq
Since \eqref{inter_ineq} holds for all $f\in F[A]$ with $f(\theta)=a$ we conclude that
\[
\Vert T_\eps a\Vert_{[B_0,B_1]_\theta}\le \max(\omega_{0,\eps},\omega_{1,\eps})\inf_{f\in F[A],\, f(\theta)=a}\Vert f\Vert_{F[A]}=\max(\omega_{0,\eps},\omega_{1,\eps})\Vert a\Vert_{[A_0,A_1]}.
\]
\end{proof}
In this paper we are mainly interested in the interpolation of Sobolev spaces. We recall that, for $s\in\R$ and $p\in(1,\infty)$, the Sobolev space $H^s_p(\R^n)$ is the set of all distributions $u\in \S'(\R^n)$ such that $\lara{D_x}^s u\in L^p$. It is a Banach space when equipped with the norm $\Vert u\Vert_{H^{s}_p}=\Vert\lara{D_x}^s u\Vert_{L^p}$. As shown in \cite[p.40]{Triebel:II}, for $p\in(1,\infty)$, $s_0,s_1\in\R$ and $\theta\in(0,1)$ one has
\beq
\label{inter_Sob}
[H^{s_0}_p,H^{s_1}_p]_\theta =H^s_p,
\eeq
with $s=(1-\theta)s_0+\theta s_1$.

\section{$L^p$ and Sobolev boundedness of pseudodifferential operators with symbol in $\mC^r_\ast S^m_{1,1}(\R^{2n})$}
\label{sec_Taylor}
In the sequel we will consider a symbol $a\in \mC^r_\ast S^m_{1,1}(\R^{2n})$, i.e., a function $a(x,\xi)$ which is smooth in $\xi$ and of class $\mC^r_\ast$ in $x$ such that the following estimates hold:
\[
\forall \alpha\in\N^n\, \exists c_\alpha>0\qquad\quad \Vert D^\alpha_\xi a(\cdot,\xi)\Vert_\infty\le c_\alpha\lara{\xi}^{m-|\alpha|},
\]
\[
\forall \alpha\in\N^n\, \exists C_\alpha>0\qquad\quad \Vert D^\alpha_\xi a(\cdot,\xi)\Vert_{\mC^r_\ast}\le C_\alpha\lara{\xi}^{m-|\alpha|+r}.
\]
Let $h\ge 0$. It is clear from Proposition \ref{prop_ast} that if we convolve $a\in \mC^r_\ast S^m_{1,1}(\R^{2n})$ with a mollifier $\rho_\eps$ we obtain a net of symbols $a_\eps(x,\xi)=(a(\cdot,\xi)\ast\rho_\eps)(x)\in \mC^{r+h}_\ast S^m_{1,1}(\R^{2n})$ such that
\[
\forall \alpha\in\N^n\, \exists c_\alpha>0\qquad\quad \Vert D^\alpha_\xi a_\eps(x,\xi)\Vert_\infty\le \Vert\rho_\eps\Vert_1 \Vert D^\alpha_\xi a(x,\xi)\Vert_\infty \le c_\alpha\lara{\xi}^{m-|\alpha|},
\]
and
\[
\forall \alpha\in\N^n\, \exists C_\alpha>0\qquad\quad \Vert D^\alpha_\xi a_\eps(\cdot,\xi)\Vert_{\mC^{r+h}_\ast}\le C_\alpha\eps^{-h}\Vert D^\alpha_\xi a(\cdot,\xi)\Vert_{\mC^r_\ast}\le C_\alpha\eps^{-h} \lara{\xi}^{m-|\alpha|+r}.
\]
We recall the following boundedness result of M. Taylor \cite[Theorem 2.1.A.]{Taylor:91}: If $r>0$ and $p\in(1,\infty)$, then for $a(x,\xi)\in \mC^r_\ast S^m_{1,1}(\R^{2n})$,
\[
a(x,D):H^{s+m,p}\to H^{s,p}
\]
provided $s\in(0,r)$.

Our goal is to drop the restriction on $s$ by working with the convolved symbol $a_\eps$ and to obtain precise continuity estimates of $a_\eps(x,D)$. We already know that since $a_\eps\in \mC^{r+h}_\ast S^m_{1,1}(\R^{2n})$ for any $h\ge 0$ the corresponding net of operators maps $H^{s+m,p}$ into $H^{s,p}$ when $s$ belongs to the interval $(0,r+h)$. This means that by convolution we are able to enlarge the $s$-interval of any positive real number $h$. A precise estimate of the Sobolev boundedness of the operator $a_\eps(x,D)$ requires a decomposition into elementary symbols and some preliminary work involving $S^1_0$-partitions of unity as in \cite[Chapter 1]{Taylor:91}.

\subsection{$S^1_0$-partition of unity and equivalent Sobolev norm}
\begin{definition}
\label{def_partition}
We say that a family of real valued smooth functions $(\psi_j)_j$ is a $S^1_0$-partition of unity (or Littlewood-Paley partition of unity) if
\begin{itemize}
\item[(i)] $\psi_0(\xi)=1$ for $|\xi|\le 1$ and $\psi_0(\xi)=0$ for $|\xi|\ge 2$;
\item[(ii)] for each $j\ge 1$,
\[
\psi_j(\xi)=\psi_0(2^{-j}\xi)-\psi_0(2^{-j+1}\xi)=\psi_1(2^{-j+1}\xi);
\]
\item[(iii)] $\sum_j\psi_j(x)=1$ for every $x\in\R^n$.
\end{itemize}
\end{definition}
Note that $\supp\, \psi_j\subseteq\{\xi:\, 2^{j-1}\le|\xi|\le 2^{j+1}\}$ for all $j\ge 1$ and by $(ii)$, for all $\alpha\in\N^n$ there exists  $c_\alpha>0$ such that
\[
\Vert D^\alpha\psi_j\Vert_\infty\le c_\alpha 2^{-j|\alpha|}
\]
for all $j\in\N$. In addition $\{\psi_j(D):\,  j\in\N\}$ and $\{\sum_{l\le j}\psi_l(D):\, j\in\N\}$ are uniformly bounded on $L^\infty$.
\begin{proposition}
\label{prop_partition}
Let $(\psi_j)_j$ be a $S^1_0$-partition of unity. Then there exists a constant $c>0$ such that
\[
\Vert \psi_j(D)f\Vert_\infty\le c\Vert f\Vert_\infty
\]
and
\[
\Vert \sum_{l\le j}\psi_l(D)f\Vert_\infty\le c\Vert f\Vert_\infty,
\]
for all $j\in\N$ and $f\in L^\infty(\R^n)$.
\end{proposition}
\begin{proof}
We begin by observing that $\psi_j(D)f(x)$ can be written as $(2\pi)^{-n} \widehat{\psi_j}\ast\wt{f}(-x)$, where $\wt{f}(x)=f(-x)$. Hence
\[
\Vert \psi_j(D)f\Vert_\infty\le  (2\pi)^{-n}\Vert f\Vert_\infty \Vert\widehat{\psi_j}\Vert_1.
\]
Since $\widehat{\psi_j}(\xi)=2^{(j-1)n}\widehat{\psi_1}(2^{j-1}\xi)$ we get
\[
\Vert \psi_j(D)f\Vert_\infty\le  (2\pi)^{-n}\Vert \widehat{\psi_1}\Vert_1\Vert f\Vert_\infty,
\]
for all $j\ge 1$. Hence
\[
\Vert \psi_j(D)f\Vert_\infty\le  (2\pi)^{-n}\max(\Vert \widehat{\psi_1}\Vert_1, \Vert \widehat{\psi_0}\Vert_1)\Vert f\Vert_\infty
\]
for all $j\in\N$. Note that $\sum_{l\le j}\psi_l(D)f=\psi_{0,j}(D)f$, with $\psi_{0,j}(\xi)=\psi_0(2^{-j}\xi)$. Arguing as above we obtain the estimate
\[
\Vert\sum_{l\le j}\psi_l(D)f\Vert_\infty\le  (2\pi)^{-n}\Vert f\Vert_\infty \Vert\widehat{\psi_{0,j}}\Vert_1\le (2\pi)^{-n}\Vert \widehat{\psi_0}\Vert_1\Vert f\Vert_\infty
\]
which completes the proof.
\end{proof}

The following technical lemmas will be employed in proving Theorem \ref{prop_elem_symb}. We refer to \cite[Theorem 2.5.6]{Triebel:I}, \cite[Lemma 1.2]{Marschall:88} and \cite[Appendix A]{Taylor:91} for the corresponding proofs.
\begin{lemma}
\label{lemma_sob}
For any $S^1_0$-partition of unity $(\psi_j)_j$, any $p\in(1,\infty)$ and $s\in\R$ the norms $\Vert\cdot\Vert_{H^{s,p}}$ and
\[
\big\Vert\{\sum_{j=0}^\infty 4^{js}|\psi_j(D)(\cdot)|^2\}^{\frac{1}{2}}\big\Vert_{L^p}
\]
are equivalent.
\end{lemma}
\begin{remark}
\label{rem_Marschall}
Note that when $(\psi_j)_j$ is a family of $\Cinfc$ functions such that $\supp\, \psi_0\subseteq\{\xi:\, |\xi|\le 2\}$ and $\supp\, \psi_j\subseteq\{\xi:\, 2^{j-1}\le|\xi|\le 2^{j+1}\}$ then there exists a constant $C>0$ such that
\[
\big\Vert\{\sum_{j=0}^\infty 4^{js}|\psi_j(D)u|^2\}^{\frac{1}{2}}\big\Vert_{L^p}\le C\Vert u\Vert_{H^{s,p}}
\]
for all $u\in H^{s,p}$. This result can be found in \cite[p.340]{Marschall:88} and is obtained by applying the multiplier theorem 2.5.6 in \cite{Triebel:I}.
\end{remark}
\begin{lemma}
\label{lemma_Taylor_2}
For any $p\in(1,\infty)$ and $s>0$ there exists a constant $C>0$ such that for all sequences $(f_k)_k$ of distributions in $\S'(\R^n)$ with $\supp\widehat{f_k}\subseteq\{\xi:\, |\xi|\le A 2^{k+1}\}$ for some $A>0$ and for all $k\in\N$, the following inequality holds:
\[
\big\Vert\sum_{k=0}^\infty f_k\big\Vert_{H^{s,p}}\le C\big\Vert\{\sum_{k=0}^\infty 4^{ks}|f_k|^2\}^{\frac{1}{2}}\big\Vert_{L^p}.
\]
\end{lemma}
We conclude this subsection by applying a $S^1_0$-partition of unity to a regularised sequence $A_{k,\eps}:=A_k\ast\rho_\eps$ of distributions $A_k$ in $\mC^r_\ast(\R^n)$. In the proof of Proposition \ref{prop_preliminary} we make  use of the fact that $\mC^r_\ast(\R^n)$ coincides with the Besov space $B^r_{\infty,\infty}(\R^n)$ (see \cite[Appendix]{GH:04} and references therein). $B^r_{\infty,\infty}(\R^n)$ is the space of all $u\in\S'(\R^n)$ such that $\Vert u\Vert_{B^r_{\infty,\infty}}:=\sup_{j\ge 0}2^{jr}\Vert\psi_j(D)u\Vert_\infty<\infty$, where $(\psi_j)_j$ is a $S^1_0$-partition of unity. The definition of $B^r_{\infty,\infty}(\R^n)$ is independent of the choice of the partition $(\psi_j)_j$. It follows that if $u\in\mC^r_\ast(\R^n)$ then
\[
2^{jr}\Vert\psi_j(D)u\Vert_\infty\le \Vert u\Vert_{B^r_{\infty,\infty}}\le c\Vert u\Vert_{\mC^{r}_\ast}.
\]
\begin{proposition}
\label{prop_preliminary}
Let $(A_k)_k$ be a sequence in $\mC^r_\ast(\R^n)$, $r>0$, and $(\psi_j)_j$ be a $S^0_1$-partition of unity. If there exists $C>0$ such that for all $k\in\N$,
$$\Vert A_k\Vert_\infty\le C$$
and
$$\Vert A_k\Vert_{\mC_\ast^r}\le C\,2^{kr}$$ then
\begin{itemize}
\item[(i)] for all $k\in\N$,
\[
\Vert A_{k,\eps}\Vert_\infty\le C,
\]
\item[(ii)] for all $h\ge 0$ there exists a constant $C'>0$ such that for all $k\in\N$ and for all $\eps\in(0,1]$
\[
\Vert A_{k,\eps}\Vert_{\mC^{r+h}_\ast}\le C'\eps^{-h}2^{kr}.
\]
\item[(iii)] Finally, for all $h\ge 0$ there exists a constant $C''>0$ such that for all $j,k\in\N$ and all $\epsilon\in(0,1]$
\[
\Vert \psi_j(D)A_{k,\eps}\Vert_\infty\le C''2^{-j(r+h)}2^{kr}\eps^{-h}.
\]
\end{itemize}
\end{proposition}
\begin{proof}
Since by definition $A_{k,\eps}=A_k\ast\rho_\eps$ we have that
\[
\Vert A_{k,\eps}\Vert_\infty\le \Vert A_k\Vert_\infty \Vert\rho_\eps\Vert_1\le C.
\]
An application of Proposition \ref{prop_ast} to $A_k$ yields
\[
\Vert A_{k,\eps}\Vert_{\mC^{r+h}_\ast}\le C_0\Vert A_k\Vert_{\mC^r_\ast}\eps^{-h},
\]
where $C_0$ does not depend on $k$ but depends on $h$. Combining this estimate with the hypothesis on $\Vert A_k\Vert_{\mC^r_\ast}$ we conclude that for all $h\ge 0$ there exists a constant $C'>0$ such that for all $k\in\N$
\[
\Vert A_{k,\eps}\Vert_{\mC^{r+h}_\ast}\le C'\eps^{-h}2^{kr}.
\]
Finally, by definition of the class $\mC^{r+h}_\ast$ we have that
\[
\Vert \psi_j(D)A_{k,\eps}\Vert_\infty\le C_1 2^{-j(r+h)}\Vert A_{k,\eps}\Vert_{\mC^{r+h}_\ast}\le C'' 2^{-j(r+h)}2^{kr}\eps^{-h},
\]
for all $k,j$ in $\N$.
\end{proof}

\subsection{$L^p$ and Sobolev boundedness of pseudodifferential operators with symbol in $\mC^r_\ast S^0_{1,1}(\R^{2n})$}
We begin by considering pseudodifferential operators with elementary symbol.
\begin{definition}
\label{def_elem}
We say that $a(x,\xi)$ is an elementary symbol in $\mC^r_\ast S^0_{1,1}(\R^{2n})$ if it is of the form
\[
\sum_{k=0}^\infty A_k(x)\varphi_k(\xi),
\]
and has the following properties:
\begin{itemize}
\item[(i)] the smooth functions $\varphi_k$ are supported on $\{\xi:\, 2^{k-1}\le|\xi|\le 2^{k+1}\}$ with $\varphi_k(\xi)=\varphi_1(2^{-k+1}\xi)$ for $k\ge 1$ and $\varphi_0$ is supported on $\{\xi:\, |\xi|\le 2\}$,
\item[(ii)] there exists a constant $C>0$ such that for all $k\in\N$,
\[
\Vert A_k\Vert_\infty\le C,\qquad\qquad \Vert A_k\Vert_{\mC_\ast^r}\le C\,2^{kr}.
\]
\end{itemize}
Analogously, $a(x,\xi)$ is an elementary symbol in $\mC^r_\ast S^m_{1,1}(\R^{2n})$ if and only if $a(x,\xi)\sharp\lara{\xi}^{-m}$ is an elementary symbol in $\mC^r_\ast S^0_{1,1}(\R^{2n})$.
\end{definition}
\begin{theorem}
\label{prop_elem_symb}
Let $r>0$ and $a(x,\xi)$ be an elementary symbol in $\mC^r_\ast S^0_{1,1}(\R^{2n})$. Let $(\psi_j)_j$ be a $S^0_1$-partition of unity and $A_{kj,\eps}:=\psi_j(D)A_{k,\eps}:=\psi_j(D)(A_k\ast\rho_\eps)$. Set
\[
a_\eps(x,\xi)=\sum_k\big\{\sum_{j=0}^{k-4}A_{kj,\eps}(x)+\sum_{j=k-3}^{k+3}A_{kj,\eps}+\sum_{j=k+4}^\infty A_{kj,\eps}(x)\big\}\varphi_k(\xi)=a_{1,\eps}(x,\xi)+a_{2,\eps}(x,\xi)+a_{3,\eps}(x,\xi).
\]
Then, the following estimates hold:
\begin{itemize}
\item[(i)] for all $s>0$ and all $p\in(1,\infty)$ there exists $C_1>0$ such that
\[
\Vert a_{1,\eps}(x,D)f\Vert_{H^{s,p}}\le C_1\Vert f\Vert_{H^{s,p}},
\]
for all $\eps\in(0,1]$ and all $f\in\S(\R^n)$;
\item[(ii)] for all $s>0$ and all $p\in(1,\infty)$ there exists $C_1>0$ such that
\[
\Vert a_{2,\eps}(x,D)f\Vert_{H^{s,p}}\le C_2\Vert f\Vert_{H^{s,p}},
\]
for all $\eps\in(0,1]$ and all $f\in\S(\R^n)$;
\item[(iii)] for all $p\in(1,\infty)$, for all $h\ge 0$ and all $s\in(0,r+h)$ there exists $C_3>0$ such that
\[
\quad\Vert a_{3,\eps}(x,D)f\Vert_{H^{s,p}}\le C_3\eps^{-h}\Vert f\Vert_{H^{s,p}}.
\]
for all $\eps\in(0,1]$ and all $f\in\S(\R^n)$.
\end{itemize}
\end{theorem}
\begin{proof}
Our proof makes use of the methods employed by M. Taylor in \cite[p.49-51]{Taylor:91}. We begin by considering $a_{1,\eps}(x,D)f=\sum_k\sum_{j=0}^{k-4}A_{kj,\eps}(x)\varphi_k(D)f$ with $f\in\S(\R^n)$. Let $f_{k,\eps}:=\sum_{j=0}^{k-4}A_{kj,\eps}(x)\varphi_k(D)f$. We can write $a_{1,\eps}(x,D)f$ as $\sum_{k=4}^\infty f_{k,\eps}$.  Since there exists $A>0$ such that $\supp(\mF(A_{kj,\eps}\varphi_k(D)f))\subseteq\{\xi:\, |\xi|\le A 2^{k+1}\}$ for all $k\ge 4$, $j=0,...,k-4$ and $\eps\in(0,1]$, an application of Lemma \ref{lemma_Taylor_2} to the sequence $(f_{k,\eps})_k$ yields
\[
\Vert a_{1,\eps}(x,D)f\Vert_{H^{s,p}}\le C\Vert\big\{\sum_{k=4}^\infty 4^{ks}\big|\sum_{j=0}^{k-4}A_{kj,\eps}\varphi_k(D)f\big|^2\big\}^{\frac{1}{2}}\Vert_{L^p}.
\]
From Proposition \ref{prop_partition} we have that $\{\sum_{l\le j}\psi_l(D):\, j\in\N\}$ is uniformly bounded on $L^\infty$ and making use of the estimate on $\Vert A_{k,\eps}\Vert_\infty$ of Proposition \ref{prop_preliminary} we conclude that there exists constants $C_0$ and $C'$ independent of $\eps$ such that
\begin{multline*}
\Vert\big\{\sum_{k=4}^\infty 4^{ks}\big|\sum_{j=0}^{k-4}A_{kj,\eps}\varphi_k(D)f\big|^2\big\}^{\frac{1}{2}}\Vert_{L^p}\le C_0\Vert \big\{\sum_{k=4}^\infty 4^{ks}\Vert A_{k,\eps}\Vert_{\infty}^2|\varphi_k(D)f|^2\big\}^{\frac{1}{2}}\Vert_{L^p}\\
\le C'\Vert \big\{\sum_{k=0}^\infty 4^{ks}|\varphi_k(D)f|^2\big\}^{\frac{1}{2}}\Vert_{L^p}.
\end{multline*}
Remark \ref{rem_Marschall} applied to $(\varphi_k(D)f)_k$ yields
\[
\Vert \big\{\sum_{k=0}^\infty 4^{ks}|\varphi_k(D)f|^2\big\}^{\frac{1}{2}}\Vert_{L^p}\le C_1\Vert f\Vert_{H^{s,p}}.
\]
Let us now take $a_{2,\eps}(x,\xi)=\sum_k\big\{\sum_{j=k-3}^{k+3}A_{kj,\eps}(x)\big\}\varphi_k(\xi)$. As above, an application of Lemma \ref{lemma_Taylor_2} combined with Proposition \ref{prop_partition} and Remark \ref{rem_Marschall}, entails
\[
\Vert a_{2,\eps}(x,D)f\Vert_{H^{s,p}}\le C\Vert\big\{\sum_{k=0}^\infty 4^{ks}\big|\sum_{j=k-3}^{k+3}A_{kj,\eps}\varphi_k(D)f\big|^2\big\}^{\frac{1}{2}}\Vert_{L^p}\le C_2\Vert f\Vert_{H^{s,p}},
\]
for some constant $C_2$ independent of $\eps$ and $f$. In order to estimate $a_{3,\eps}(x,D)f$ we recall that by Proposition \ref{prop_preliminary} for all $h\ge 0$ there exists $C'>0$ such that
\[
\Vert \psi_j(D)A_{k,\eps}\Vert_\infty\le C' 2^{-j(r+h)}2^{kr}\eps^{-h},
\]
for all $j,k$ and for all $\eps\in(0,1]$. From Lemma \ref{lemma_Taylor_2} we have
\begin{multline*}
\Vert a_{3,\eps}(x,D)f\Vert_{H^{s,p}}=\Vert\sum_{k=0}^\infty\sum_{j=k+4}^\infty A_{kj,\eps}\varphi_k(D)f\Vert_{H^{s,p}}=\Vert\sum_{j=4}^\infty\sum_{k=0}^{j-4} A_{kj,\eps}\varphi_k(D)f\Vert_{H^{s,p}}\\
\le C\Vert\big\{\sum_{j=4}^\infty 4^{js}\big|\sum_{k=0}^{j-4}A_{kj,\eps}\varphi_k(D)f\big|^2\big\}^{\frac{1}{2}}\Vert_{L^p}\le C'\eps^{-h}\Vert\big\{\sum_{j=4}^\infty 4^{j(s-r-h)}\big(\sum_{k=0}^{j-4}2^{k(r+h)}|\varphi_k(D)f|\big)^2\big\}^{\frac{1}{2}}\Vert_{L^p}.
\end{multline*}
Since
\begin{multline*}
\sum_{j=4}^\infty 4^{j(s-r-h)}\big(\sum_{k=0}^{j-4}2^{k(r+h)}|\varphi_k(D)f|\big)^2=\sum_{j=4}^\infty \big(\sum_{k=0}^{j-4}2^{(k-j)(r+h-s)}2^{ks}|\varphi_k(D)f|\big)^2\\
\le 2\sum_{j=4}^\infty \sum_{k=0}^{j-4}2^{(k-j)(r+h-s)}4^{ks}|\varphi_k(D)f|^2,
\end{multline*}
changing order in the sums, we get
\[
\Vert\big\{\sum_{j=4}^\infty 4^{j(s-r-h)}\big(\sum_{k=0}^{j-4}2^{k(r+h)}|\varphi_k(D)f|\big)^2\big\}^{\frac{1}{2}}\Vert_{L^p}\le \sqrt{2}\Vert\big\{\sum_{j=0}^\infty 4^{j(s-r-h)}\sum_{k=0}^\infty 4^{ks}|\varphi_k(D)f|^2\big\}^{\frac{1}{2}}\Vert_{L^p}.
\]
Hence, if $0<s<r+h$ there exists a constant $C''$ such that
\[
C'\eps^{-h}\Vert\big\{\sum_{j=4}^\infty 4^{j(s-r-h)}\big(\sum_{k=0}^{j-4}2^{k(r+h)}|\varphi_k(D)f|\big)^2\big\}^{\frac{1}{2}}\Vert_{L^p}\le C''\eps^{-h}\Vert\big\{\sum_{k=0}^\infty 4^{ks}|\varphi_k(D)f|^2\big\}^{\frac{1}{2}}\Vert_{L^p}.
\]
Again by Remark \ref{rem_Marschall} we conclude that there exists $C_3>0$ such that for all $\eps\in(0,1]$ and $f\in\S(\R^n)$,
\[
\Vert a_{3,\eps}(x,D)f\Vert_{H^{s,p}}\le C_3\eps^{-h}\Vert f\Vert_{H^{s,p}}.
\]
\end{proof}
\begin{corollary}
\label{cor_cont}
Let $a(x,\xi)$ be an elementary symbol in $\mC^r_\ast S^m_{1,1}(\R^{2n})$. If $r>0$ and $p\in(1,\infty)$ then for all $h\ge 0$ and all $s\in(0,r+h)$ there exists $C>0$ such that
\[
\Vert a_\eps(x,D)f\Vert_{H^{s,p}}\le C\eps^{-h}\Vert f\Vert_{H^{s+m,p}}.
\]
for all $\eps\in(0,1]$ and $f\in{H^{s+m,p}}(\R^n)$.
\end{corollary}
\begin{proof}
We begin by writing $\Vert a_\eps(x,D)f\Vert_{H^{s,p}}$ as
\beq
\label{eq_1}
\Vert (a_\eps(x,D)\lara{D_x}^{-m})\lara{D_x}^m f\Vert_{H^{s,p}}.
\eeq
Observing that $(a(\cdot,\xi)\ast\rho_\eps)\sharp\lara{\xi}^{-m}=(a(\cdot,\xi))\sharp\lara{\xi}^{-m})\ast\rho_\eps$ with $a(\cdot,\xi))\sharp\lara{\xi}^{-m}$ elementary symbol in  $\mC^r_\ast S^0_{1,1}(\R^{2n})$ by applying Theorem \ref{prop_elem_symb} we have that if $r>0$ and $p\in(1,\infty)$ then for all $h\ge 0$ and all $s\in(0,r+h)$ there exists $C>0$ such that
\beq
\label{eq_2}
\Vert (a_\eps(x,D)\lara{D_x}^{-m})\lara{D_x}^m f\Vert_{H^{s,p}}\le C\eps^{-h}\Vert\lara{D_x}^m f\Vert_{H^{s,p}}.
\eeq
Combining \eqref{eq_1} with \eqref{eq_2} we conclude that for all $h\ge 0$ and all $s\in(0,r+h)$ there exists $C>0$ such that
\[
\Vert a_\eps(x,D)f\Vert_{H^{s,p}}\le C\eps^{-h}\Vert f\Vert_{H^{s+m,p}}
\]
for all $\eps\in(0,1]$ and all $f\in{H^{s+m,p}}(\R^n)$.
\end{proof}
It is well-known that a symbol $a\in \mC^r_\ast S^0_{1,1}(\R^{2n})$ can be decomposed into a sum of elementary symbols. More precisely, referring to \cite[p. 48-49]{Taylor:91} and \cite[p. 18-20]{Taylor:00}, we have that
\[
a(x,\xi)=\sum_{\nu=1}^\infty c_\nu\sum_{k=0}^\infty a_k^\nu(x)\varphi_k^\nu(\xi),
\]
where, the sequence $c_\nu$ is rapidly decreasing and $\sum_{k=0}^\infty a_k^\nu(x)\varphi_k^\nu(\xi)$ is an elementary symbol. In addition there exists a constant $c>0$ such that
\beq
\label{est_elem_imp}
\Vert a_k^\nu\Vert_\infty\le c,\qquad\qquad \Vert a^\nu_k\Vert_{\mC_\ast^r}\le c\,2^{kr}
\eeq
for all values of $k$ and $\nu$. Passing to the regularisation via convolution with a mollifier $\rho_\eps$ we easily see that
\[
a_\eps(x,\xi)=(a(\cdot,\xi)\ast\rho_\eps)(x)=\sum_\nu c_\nu\sum_k a_{k,\eps}^\nu(x)\varphi_{k}^\nu(\xi),
\]
where $a^\nu_{k,\eps}(x)=a^\nu_k\ast\rho_\eps(x)$. We are now ready to prove the following theorem.
\begin{theorem}
\label{cor_cont_1}
Let $a(x,\xi)$ be a symbol in $\mC^r_\ast S^m_{1,1}(\R^{2n})$. If $r>0$ and $p\in(1,\infty)$ then for all $h\ge 0$ and all $s\in(0,r+h)$ there exists $C>0$ such that
\[
\Vert a_\eps(x,D)f\Vert_{H^{s,p}}\le C\eps^{-h}\Vert f\Vert_{H^{s+m,p}}
\]
for all $\eps\in(0,1]$ and $f\in{H^{s+m,p}}(\R^n)$.
\end{theorem}
\begin{proof}
It is not restrictive to assume that $a$ has order $0$. Making use of the decomposition into elementary symbols above we concentrate on
\[
\sum_k a_{k,\eps}^\nu(x)\varphi_{k}^\nu(\xi),
\]
where we can assume that $\Vert \varphi_0^\nu\Vert_1$ and $\Vert\varphi_1^\nu\Vert_1$ do not depend on $\nu$. We recall that in the estimates \eqref{est_elem_imp} the constant $c$ does not depend on $\nu$ and $k$. An investigation of the proof of Theorem \ref{prop_elem_symb}, in which we make use of the results of Propositions \ref{prop_partition} and \ref{prop_preliminary}, shows that if $r>0$ and $p\in(1,\infty)$ then for all $h\ge 0$ and all $s\in(0,r+h)$ there exists $C_1>0$, independent of $\nu$ and $k$, such that
\[
\Vert a^\nu_{k,\eps}(x,D)f\Vert_{H^{s,p}}\le C_1\eps^{-h}\Vert f\Vert_{H^{s+m,p}}.
\]
for all $\eps\in(0,1]$ and $f\in{H^{s+m,p}}(\R^n)$. Since the sequence $c_\nu$ is rapidly decreasing in $\nu$ we can conclude that for each $h\ge 0$ and $s\in(0,r+h)$ there exists a constant $C>0$ such that
\[
\Vert a_\eps(x,D)f\Vert_{H^{s,p}}\le C\eps^{-h}\Vert f\Vert_{H^{s+m,p}}
\]
for all $\eps\in(0,1]$.
\end{proof}


We conclude this section with the following continuity result for pseudodifferential operators with regular symbol of type $(1,\delta)$.
\begin{proposition}
\label{prop_Sm1delta}
Let $a\in S^m_{1,\delta}(\R^{2n})$ with $\delta\in[0,1)$. If $p\in(1,\infty)$ then for all $s\in\R$ there exists $C>0$ such that
\[
\Vert a_\eps(x,D)f\Vert_{H^{s,p}}\le C\Vert f\Vert_{H^{s+m,p}}
\]
for all $\eps\in(0,1]$.
\end{proposition}
\begin{proof}
We begin by observing that $S^m_{1,\delta}(\R^{2n})\subseteq \mC^k_\ast S^m_{1,1}(\R^{2n})$ for all $k\in\N$. This is due to the fact that the space of continuous and bounded functions with continuous and bounded derivatives up to order $k$ is contained in $\mC^k_\ast(\R^n)$ (see \cite[Appendix A]{Taylor:91}). In detail,
\[
\Vert D^\alpha_\xi a(\cdot,\xi)\Vert_{\mC^k_\ast}\le C\sup_{|\beta|\le k}\Vert D^\alpha_\xi D^\beta_x a(x,\xi)\Vert_{L^\infty(\R^n_x)}\le C_\alpha\lara{\xi}^{m-|\alpha|+\delta k}\le C_\alpha\lara{\xi}^{m-|\alpha|+k}
\]
and from the definition of $S^m_{1,\delta}(\R^{2n})$,
\[
\Vert D^\alpha_\xi a(\cdot,\xi)\Vert_\infty\le c_\alpha\lara{\xi}^{m-|\alpha|}.
\]
We can therefore fix $s\in(0,k)$ and apply Theorem \ref{cor_cont_1} with $h=0$. We obtain that
\[
\Vert a_\eps(x,D)f\Vert_{H^{s,p}}\le C\Vert f\Vert_{H^{s+m,p}}.
\]
Making now $k$ vary in $\N$ we conclude that the previous mapping property holds for all $s\in(0,+\infty)$ with $C$ depending on $s$. Let us consider the transposed operator ${\,}^t a_\eps(x,D)$ of $a_\eps(x,D)$. We can write
\[
{\,}^t(a_\eps(x,D))f=\wt{a}_\eps(x,D)f,
\]
where $\wt{a}(x,\xi)=a(x,-\xi)$. Arguing as above and applying Theorem \ref{cor_cont_1} to $\wt{a}_\eps(x,D)$ we have that for all $s>0$ there exists a constant $C>0$ independent of $\eps$ such that for all $\eps\in(0,1]$,
\[
\Vert {\,}^t a_\eps(x,D)f\Vert_{H^{s,p}}\le C\Vert f\Vert_{H^{s+m,p}}.
\]
Since $(H^{s,p'}(\R^n))'=H^{-s,p}(\R^n)$ with $1/p+1/p'=1$ by duality methods we obtain that
\begin{multline*}
\Vert a_\eps(x,D)f\Vert_{H^{-s-m,p}}\le \sup_{\Vert g\Vert_{H^{s+m,p'}}\le 1}|\lara{a_\eps(x,D)f,g}|= \sup_{\Vert g\Vert_{H^{s+m,p'}}\le 1}|\lara{f,{\,}^t a_\eps(x,D)g}|\\
\le \Vert f\Vert_{H^{-s,p}}\Vert{\,}^t a_\eps(x,D)g\Vert_{H^{s,p'}}\le C\Vert f\Vert_{H^{-s,p}}.
\end{multline*}
This means that for all $s<0$,
\[
\Vert a_\eps(x,D)f\Vert_{H^{s-m,p}}\le C\Vert f\Vert_{H^{s,p}},
\]
or in other words,
\[
\Vert a_\eps(x,D)f\Vert_{H^{s,p}}\le C\Vert f\Vert_{H^{s+m,p}}.
\]
for $s<-m$.\\
We now take the interpolation couples $\{H^{s_0+m}_p,H^{s_1+m}_p\}$ and $\{H^{s_0}_p,H^{s_1}_p\}$, with $s_0=-m-\lambda$, $s_1=\lambda$, $\lambda>0$. The net of operators $a_\eps(x,D)$ maps $H^{s_j+m}_p$ into $H^{s_j}_p$ for $j=0,1$ and fulfills the hypothesis of Proposition \ref{prop_interpolation} with $\omega_{0,\eps}=c_0>0$ and $\omega_{1,\eps}=c_1>0$. Making use of \eqref{inter_Sob} and of Proposition \ref{prop_interpolation} we conclude that
\[
\Vert a_\eps(x,D)f\Vert_{H^{s+m}_p}\le \max(c_0,c_1)\Vert f\Vert_{H^{s,p}}
\]
for all $\eps\in(0,1]$ and for $s=(1-\theta)s_0+\theta s_1$ with $\theta\in(0,1)$. This means that also for $s\in[-m,0]$ (if $m>0$) and for $s\in[0,-m]$ (if $m<0$) there exists some constant $C>0$ such that
\[
\Vert a_\eps(x,D)f\Vert_{H^{s,p}}\le C\Vert f\Vert_{H^{s+m,p}}
\]
for all $\eps\in(0,1]$.
\end{proof}

\section{$L^p$ and Sobolev boundedness of nets of pseudodifferential operators with regular symbol}
This section is devoted to nets of pseudodifferential operators $a_\eps(x,D)$ with regular symbol, i.e. $a_\eps\in S^m(\R^{2n})$, depending on the parameter $\eps\in(0,1]$. We recall that $a\in S^m(\R^{2n})$ if and only if
\[
|a|^{(m)}_{\alpha,\beta}=\sup_{(x,\xi)\in\R^{2n}}\lara{\xi}^{-m+|\alpha|}|\partial^\alpha_\xi\partial^\beta_x a(x,\xi)|<\infty,
\]
for all $\alpha,\beta\in\N^n$. Differently from the previous sections, we do not require that $a_\eps$ is generated by a symbol $a$ via convolution with a mollifier $\rho_\eps$. We therefore consider a wider class of nets $a_\eps(x,D)$ with respect to Section \ref{sec_Taylor}. 
We want to investigate the $L^p$ and Sobolev boundedness of a net of pseudodifferential operators $a_\eps(x,D)$ with $(a_\eps)_\eps\in S^m(\R^{2n})^{(0,1]}$. This requires the following lemmas whose proof can be found in \cite[Lemma 10.9]{Wong:99}, \cite[Theorem 2.5]{Hoermander:60} and \cite[Lemma 10.10]{Wong:99}.
\begin{lemma}
\label{lemma_1}
Let $Q_0$ be the cube with center at the origin and edges of length $1$ parallel to the coordinates axes in $\R^n$. Let $\eta\in\Cinfc(\R^n)$ be identically $1$ on $Q_0$. Let $a\in S^0(\R^n)$, $a_m(x,\xi)=\eta(x-m)a(x,\xi)$ for $m\in \Z^n$ and
\[
\widehat{a}_m(\lambda,\xi)=\int_{\R^n}\esp^{-i\lambda x}a_m(x,\xi)\, dx.
\]
Then, for all $\alpha\in\N^n$ and all $N\in\N$ there exists $C>0$ depending only on $n$, $\eta$ and $N$ such that
\[
|D^\alpha_\xi\widehat{a}_m(\lambda,\xi)|\le C\sup_{|\beta|\le N}|a|^{(0)}_{\alpha,\beta}\,\lara{\xi}^{-|\alpha|}\lara{\lambda}^{-N},
\]
for all $(\lambda,\xi)\in\R^{2n}$.
\end{lemma}
\begin{lemma}
\label{lemma_2}
Let $f\in \mC^k(\R^n\setminus 0)$, $k>n/2$, be such that there exists $B>0$ for which
\[
|D^\alpha f(\xi)|\le B|\xi|^{-|\alpha|},\qquad \xi\neq 0,
\]
for all $\alpha\in\N^n$ with $|\alpha|\le k$. Then, for all $p\in(1,\infty)$ there exists $C>0$ depending only on $p$ and $n$, such that
\[
\Vert f(D)(\varphi)\Vert_p\le CB\Vert\varphi\Vert_p,
\]
for all $\varphi\in\S(\R^n)$.
\end{lemma}
\begin{lemma}
\label{lemma_3}
Let $a\in S^0(\R^{2n})$ and $K_a$ be the distribution $\mF^{-1}_{\xi,\to z}(a(x,\xi))$ in $\S'(\R^{2n})$. Then,
\begin{itemize}
\item[(i)] for each $x\in\R^n$, $K_a(x,\cdot)$ is a function defined on $\R^n\setminus 0$,
\item[(ii)] for each $N$ sufficiently large there exists a constant $c$, depending only on $N$ and $n$ such that
\[
|K_a(x,z)|\le c\sup_{|\alpha|\le N}|a|^{(0)}_{\alpha,0}\, |z|^{-N}
\]
for all $z\neq 0$,
\item[(iii)] for each $x\in\R^n$ and $\varphi\in\S(\R^n)$ vanishing in a neighborhood of $x$,
\[
a(x,D)\varphi(x)=\int_{\R^n}K_a(x,x-z)\varphi(z)\, dz.
\]
\end{itemize}
\end{lemma}
\begin{theorem}
\label{theo_Lp_wong}
Let $(a_\eps)_\eps\in S^0(\R^{2n})^{(0,1]}$ and $p\in(1,\infty)$. Then, there exists $N\in\N$ and a constant $C$ depending only on $n$, $N$ and $p$ such that
\[
\Vert a_\eps(x,D)\varphi\Vert_p\le C\sup_{|\alpha+\beta|\le N}|a_\eps|^{(0)}_{\alpha,\beta}\,\Vert\varphi\Vert_p,
\]
for all $\varphi\in \S(\R^n)$ and $\eps\in(0,1]$.
\end{theorem}
\begin{proof}
As in the proof of Theorem 10.7 in \cite{Wong:99} we write $\R^{n}$ as a union of cubes $Q_m$, where $Q_m$ is the cube with center $m\in \Z^m$ and edges of length $1$ which are parallel to the coordinate axes. Let $Q^{\ast}_m$ and $Q^{\ast\ast}_m$ be cubes with center $m$ and edges parallel to the coordinate axes with length $3/2$ and $2$, respectively. It follows that $Q_m\subset Q_m^\ast\subset Q^{\ast\ast}_m$ and that for some $\delta>0$ one has $|x-z|\ge\delta$ for all $x\in Q_m$ and $z\in\R^n\setminus Q_m^\ast$.

Let now $\psi\in\Cinfc(\R^n)$ be such that $0\le\psi\le 1$, $\supp\psi\subseteq Q_0^{\ast\ast}$ and $\psi(x)=1$ on a neighborhood of $Q_0^\ast$. It follows that $\psi_m(x)=\psi(x-m)$ has support contained in $Q_m^{\ast\ast}$ and $\psi_m(x)=1$ on a neighborhood of $Q_m^\ast$. For each $\varphi\in\S(\R^n)$ we can write $\varphi=\varphi_{1,m}+\varphi_{2,m}$, where $\varphi_{1,m}=\psi_m\varphi$ and $\varphi_{2,m}=(1-\psi_m)\varphi$, and then
\[
a_\eps(x,D)\varphi=a_\eps(x,D)\varphi_{1,m}+a_\eps(x,D)\varphi_{2,m}.
\]
It is clear that
\begin{multline}
\label{decomp_Wong}
\Vert a_\eps(x,D)\varphi\Vert_p^p=\sum_{m\in\Z^m}\int_{Q_m}|a_\eps(x,D)\varphi(x)|^p\, dx\\
\le 2^p\biggl(\sum_{m\in\Z^n}\int_{Q_m}|a_\eps(x,D)\varphi_{1,m}(x)|^p\, dx +\int_{Q_m}|a_\eps(x,D)\varphi_{2,m}(x)|^p\, dx\biggr).
\end{multline}
Our proof consists of three steps:
\begin{enumerate}
\item estimate of $\int_{Q_m}|a_\eps(x,D)\varphi_{1,m}(x)|^p\, dx$,
\item estimate of $\int_{Q_m}|a_\eps(x,D)\varphi_{2,m}(x)|^p\, dx$,
\item combination of $1$ and $2$.
\end{enumerate}

\bf{Step 1.}\rm\, We begin by considering \[
\int_{Q_m}|a_\eps(x,D)v(x)|^p\, dx,
\]
where $v\in \S(\R^n)$. Let $\eta\in\Cinfc(\R^n)$ be identically $1$ on $Q_0$ and $a_m(x,\xi)=\eta(x-m)a(x,\xi)$. Hence,
\beq
\label{ineq_correc}
\int_{Q_m}|a_\eps(x,D)v(x)|^p\, dx\le \int_{\R^n}|a_{m,\eps}(x,D)v(x)|^p\, dx.
\eeq
Since $a_{m,\eps}$ is compactly supported in $x$ we can write $a_{m,\eps}(x,D)v(x)$ as
\[
\int_{\R^n}\esp^{ix\lambda}\int_{\R^n}\esp^{ix\xi}\,\widehat{a}_{m,\eps}(\lambda,\xi)\widehat{v}(\xi)\, \dslash\xi\, \dslash\lambda =\int_{\R^n}\esp^{ix\lambda}\,\widehat{a}_{m,\eps}(\lambda,D)(v)(x)\, \dslash\lambda
\]
From Lemma \ref{lemma_1} we have that for all $N\in\N$
\[
|D^\alpha_\xi\widehat{a}_{m,\eps}(\lambda,\xi)|\le C\sup_{|\beta|\le N}|a_{m,\eps}|^{(0)}_{\alpha,\beta}\,\lara{\xi}^{-|\alpha|}\lara{\lambda}^{-N},
\]
where $C$ depends only on $n$, $\eta$ and $N$. We can therefore apply Lemma \ref{lemma_2} to $f(\xi)=\widehat{a}_{m,\eps}(\lambda,\xi)$ with
\[
B=C\sup_{|\beta|\le N, |\alpha|\le \lfloor n/2\rfloor+1}|a_{m,\eps}|^{(0)}_{\alpha,\beta}\,\lara{\lambda}^{-N}
\]
and obtain that there exists a constant $C'$, depending on $n,N, \eta$ and $p$ such that
\beq
\label{est_m_1}
\Vert \widehat{a}_{m,\eps}(\lambda,D)(v)(x)\Vert_{L^p(\R^n_x)}\le C'\sup_{|\beta|\le N, |\alpha|\le \lfloor n/2\rfloor+1}|a_{m,\eps}|^{(0)}_{\alpha,\beta}\ \lara{\lambda}^{-N}\Vert v\Vert_p
\eeq
for all $\lambda\in\R^n$, for all $\eps\in (0,1]$, for all $m\in\Z^n$ and for all $v\in\S(\R^n)$. An application of the Minkowski's inequality in integral form leads from \eqref{est_m_1} to
\begin{multline*}
\Vert {a}_{m,\eps}(\lambda,D)(v)\Vert_p=\biggl\{\int_{\R^n}\biggl|\int_{\R^n} \esp^{ix\lambda}\,\widehat{a}_{m,\eps}(\lambda,D)(v)(x)\, \dslash\lambda\biggr|^p dx\biggr\}^{\frac{1}{p}}\le\int_{\R^n}\biggl\{\int_{\R^n}|\widehat{a}_{m,\eps}(\lambda,D)(v)(x)|^p dx\biggr\}^{\frac{1}{p}}\dslash\lambda\\
=\int_{\R^n}\Vert \widehat{a}_{m,\eps}(\lambda,D)(v)\Vert_p\, \dslash\lambda\le C'\sup_{|\beta|\le N, |\alpha|\le \lfloor n/2\rfloor+1}|a_{m,\eps}|^{(0)}_{\alpha,\beta}\ \int_{\R^n}\lara{\lambda}^{-N}\, \dslash\lambda\,\Vert v\Vert_p.
\end{multline*}
Thus, choosing $N=n+1$ we get
\beq
\label{est_m_2}
\Vert {a}_{m,\eps}(\lambda,D)(v)\Vert_p\le C'\sup_{\substack{|\beta|\le n+1,\\ |\alpha|\le \lfloor n/2\rfloor+1}}|a_{m,\eps}|^{(0)}_{\alpha,\beta}\ \Vert v\Vert_p,
\eeq
valid for all $m\in\Z^n$, for all $\eps\in(0,1]$ and $v\in\S(\R^n)$. Going back to $\int_{Q_m}|a_\eps(x,D)\varphi_{1,m}(x)|^p\, dx$, the estimate \eqref{est_m_2} combined with \eqref{ineq_correc} yields
\beq
\label{est_m_3}
\int_{Q_m}|a_\eps(x,D)\varphi_{1,m}(x)|^p\, dx\le \Vert {a}_{m,\eps}(\lambda,D)(\varphi_{1,m})\Vert_p^p\le C_p\big(\sup_{\substack{|\beta|\le n+1,\\ |\alpha|\le \lfloor n/2\rfloor+1}}|a_{m,\eps}|^{(0)}_{\alpha,\beta}\big)^p \Vert \varphi_{1,m}\Vert_p^p,
\eeq
where $C_p$ does not depend on $m$ and $\eps$.

\bf{Step 2.}\rm\, We now want to estimate $\int_{Q_m}|a_\eps(x,D)\varphi_{2,m}(x)|^p\, dx$. We start by studying $|a_\eps(x,D)\varphi_{2,m}(x)|$ when $x\in Q_m$. Since $\varphi_{2,m}$ is identically $0$ on $Q^\ast_m\supset Q_m$ from Lemma \ref{lemma_3} we have
\beq
\label{est_p_1}
|a_\eps(x,D)\varphi_{2,m}(x)|=\biggl|\int_{\R^n}K_{a_\eps}(x,x-z)\varphi_{2,m}(z)\, dz\biggr|\le c\sup_{|\alpha|\le  2N}|a_\eps|^{(0)}_{\alpha,0}\int_{\R^n\setminus Q^\ast_m}|x-z|^{-2N}|\varphi_{2,m}(z)|\, dz,
\eeq
valid for $2N>n$ and for all $x\in Q_m$ with some constant $c$ depending only on $n$ and $N$. Let us fix $\lambda\ge\sqrt{n}+1$. Since $|x-z|\ge \delta$ for all $x\in Q_m$ and all $z\in\R^n\setminus Q_m^\ast$, there exists a constant $C_{\lambda,N}$ such that
\beq
\label{est_p_2}
\frac{|x-z|^{-2N}}{(\lambda+|x-z|)^{-2N}}\le C_{\lambda,N}
\eeq
on the same domain and, for all $x\in Q_m$,
\beq
\label{est_p_3}
\lambda+|x-z|\ge\lambda+|m-z|-|x-m|\ge \big(\lambda-\frac{\sqrt{n}}{2}\big)+|m-z|\ge \frac{\sqrt{n}}{2}+1+|m-z|=\mu+|m-z|.
\eeq
Combining \eqref{est_p_1} with \eqref{est_p_2} and \eqref{est_p_3} we get the estimate
\begin{multline}
\label{est_p_4}
|a_\eps(x,D)\varphi_{2,m}(x)|\le c \sup_{|\alpha|\le  2N}|a_\eps|^{(0)}_{\alpha,0}\int_{\R^n\setminus Q^\ast_m}C_{\lambda,N}(\lambda+|x-x|)^{-2N}|\varphi_{2,m}(z)|\, dz\\
\le c\, C_{\lambda,N}\sup_{|\alpha|\le  2N}|a_\eps|^{(0)}_{\alpha,0}\int_{\R^n\setminus Q^\ast_m}\frac{(\mu+|x-z|)^{-N}}{(\mu+|m-z|)^{N}}|\varphi_{2,m}(z)|\, dz,
\end{multline}
valid for all $x\in Q_m$ and all $\eps\in(0,1]$. By Minkowski's inequality and H\"older's inequality we can write
\begin{multline*}
\biggl(\int_{Q_m}|a_\eps(x,D)\varphi_{2,m}(x)|^p\, dx\biggr)^{\frac{1}{p}}\\
\le c\, C_{\lambda,N}\sup_{|\alpha|\le  2N}|a_\eps|^{(0)}_{\alpha,0}\int_{\R^n\setminus Q_m^\ast}\biggl\{\int_{Q_m}\frac{(\mu+|x-z|)^{-Np}}{(\mu+|m-z|)^{Np}}|\varphi_{2,m}(z)|^p\,dx\biggr\}^{\frac{1}{p}}dz\\
=c\, C_{\lambda,N}\sup_{|\alpha|\le  2N}|a_\eps|^{(0)}_{\alpha,0}\int_{\R^n\setminus Q_m^\ast}\frac{|\varphi_{2,m}(z)|}{(\mu+|m-z|)^{N}}\biggl\{\int_{Q_m}(\mu+|x-z|)^{-Np}\,dx\biggr\}^{\frac{1}{p}}dz\\
=C_{\lambda,N,p}\sup_{|\alpha|\le  2N}|a_\eps|^{(0)}_{\alpha,0}\int_{\R^n\setminus Q_m^\ast}\frac{|\varphi_{2,m}(z)|}{(\mu+|m-z|)^{N}}\, dz\\
\le C_{\lambda,N,p}\sup_{|\alpha|\le  2N}|a_\eps|^{(0)}_{\alpha,0}\biggl\{\int_{\R^n\setminus Q_m^\ast}{(\mu+|m-z|)^{\frac{-Np'}{2}}}\, dz\biggr\}^{\frac{1}{p'}}\biggl\{\int_{\R^n\setminus Q_m^\ast}\frac{|\varphi_{2,m}(z)|^p}{(\mu+|m-z|)^{\frac{Np}{2}}}\, dz\biggr\}^{\frac{1}{p}}.
\end{multline*}
At this point choosing $N$ large enough ($Np'/2>n$) we obtain that there exists a constant $C_{\lambda,N,p}$, depending only on $\lambda$, $N$ and $p$ such that
\beq
\label{est_p_5}
\int_{Q_m}|a_\eps(x,D)\varphi_{2,m}(x)|^p\, dx\le C_{\lambda,N,p}\big(\sup_{|\alpha|\le  2N}|a_\eps|^{(0)}_{\alpha,0}\big)^p\int_{\R^n\setminus Q_m^\ast}\frac{|\varphi_{2,m}(z)|^p}{(\mu+|m-z|)^{\frac{Np}{2}}}\, dz,
\eeq
for all $m\in\Z^n$ and $\eps\in(0,1]$.

\bf{Step 3.}\rm\, A combination of \eqref{decomp_Wong} with \eqref{est_m_3} and \eqref{est_p_5} yields
\begin{multline*}
\Vert a_\eps(x,D)\varphi\Vert_p^p\le 2^p C_p\sum_{m\in\Z^n}\big(\sup_{\substack{|\beta|\le n+1,\\ |\alpha|\le \lfloor n/2\rfloor+1}}|a_{m,\eps}|^{(0)}_{\alpha,\beta}\big)^p \Vert \varphi_{1,m}\Vert_p^p\\
+ 2^pC_{\lambda,N,p}\big(\sup_{|\alpha|\le  2N}|a_\eps|^{(0)}_{\alpha,0}\big)^p\sum_{m\in\Z^n}\int_{\R^n\setminus Q_m^\ast}\frac{|\varphi_{2,m}(z)|^p}{(\mu+|m-z|)^{\frac{Np}{2}}}\, dz,
\end{multline*}
with $\lambda\ge\sqrt{n}+1$ and $Np> 2n(p-1)$. From the definition of $a_{m,\eps}$, $\varphi_{1,m}$ and $\varphi_{2,m}$ we get (for some new constant $C_p$),
\begin{multline}
\label{fin_est}
\Vert a_\eps(x,D)\varphi\Vert_p^p\le 2^p C_p\,\big(\sup_{\substack{|\beta|\le n+1,\\ |\alpha|\le \lfloor n/2\rfloor+1}}|a_{\eps}|^{(0)}_{\alpha,\beta}\big)^p \sum_{m\in\Z^n}\int_{Q_m^{\ast\ast}}|\varphi(x)|^p\, dx\\
+ 2^pC_{\lambda,N,p}\big(\sup_{|\alpha|\le  2N}|a_\eps|^{(0)}_{\alpha,0}\big)^p\sum_{m\in\Z^n}\int_{\R^n\setminus Q_m^\ast}\frac{|\varphi_{2,m}(z)|^p}{(\mu+|m-z|)^{\frac{Np}{2}}}\, dz\\
\le 2^p C_p\,\big(\sup_{\substack{|\beta|\le n+1,\\ |\alpha|\le \lfloor n/2\rfloor+1}}|a_{\eps}|^{(0)}_{\alpha,\beta}\big)^p \sum_{m\in\Z^n}\int_{Q_m^{\ast\ast}}|\varphi(x)|^p\, dx\\
+ 2^pC_{\lambda,N,p}\big(\sup_{|\alpha|\le  2N}|a_\eps|^{(0)}_{\alpha,0}\big)^p\sum_{m\in\Z^n}\sum_{l\neq m}\int_{Q_l}\frac{|\varphi_{2,m}(z)|^p}{(\mu+|m-z|)^{\frac{Np}{2}}}\, dz.
\end{multline}
Arguing as in \eqref{est_p_3} we have that $\mu+|m-z|\ge 1+|m-l|$ when $z\in Q_l$ with $l\neq m$. Hence
\begin{multline*}
\sum_{m\in\Z^n}\sum_{l\neq m}\int_{Q_l}\frac{|\varphi_{2,m}(z)|^p}{(\mu+|m-z|)^{\frac{Np}{2}}}\, dz\le \sum_{m\in\Z^n}\sum_{l\neq m}(1+|m-l|)^{-\frac{Np}{2}}\int_{Q_l}{|\varphi_{2,m}(z)|^p}\, dz\\
\le\sum_{m\in\Z^n}\sum_{l\in\Z^n}(1+|m-l|)^{-\frac{Np}{2}}\int_{Q_l}{|\varphi_{2,m}(z)|^p}\, dz\le\sum_{m\in\Z^n}(1+|m|)^{-\frac{Np}{2}}\sum_{l\in\Z^n}\int_{Q_l}{|\varphi(z)|^p}\, dz.
\end{multline*}
At this point choosing $Np>\max(2n(p-1),2n)$ and going back to \eqref{fin_est} we obtain the estimate
\beq
\label{fin_est_1}
\Vert a_\eps(x,D)\varphi\Vert_p^p\le C_{p,n,N}\biggl(\big(\sup_{\substack{|\beta|\le n+1,\\ |\alpha|\le \lfloor n/2\rfloor+1}}|a_{\eps}|^{(0)}_{\alpha,\beta}\big)^p+\big(\sup_{|\alpha|\le  2N}|a_\eps|^{(0)}_{\alpha,0}\big)^p\biggr)\Vert \varphi\Vert_p^p,
\eeq
valid for all $\eps\in(0,1]$ and $\varphi\in\S(\R^n)$. This completes the proof.
\end{proof}
\begin{remark}
We recall that a net of symbols $(a_\eps)_\eps$ in $S^m(\R^{2n})$ is \emph{moderate} if for all $\alpha,\beta\in\N^n$ there exists $N\in\N$ such that
\[
|a_\eps|^{(m)}_{\alpha,\beta}=O(\eps^{-N})
\]
as $\eps\to 0$. This is the typical representative of a generalised symbol in the Colombeau framework as defined in \cite{Garetto:ISAAC07, GGO:03}. Theorem \ref{theo_Lp_wong} shows that the net $(a_\eps(x,D)\varphi)_\eps$ has in the norm $\Vert\cdot\Vert_p$ the same kind of dependence on $\eps$ of the symbol $(a_\eps)_\eps$. It follows that, via action of the corresponding pseudodifferential operator, moderate nets of symbols provide moderate nets of $L^p$ functions.
\end{remark}
\begin{corollary}
\label{corol_Lp_wong}
Let $(a_\eps)_\eps\in S^m(\R^{2n})^{(0,1]}$ and $p\in(1,\infty)$. Then, for all $s\in\R$ there exists $N\in\N$ and a constant $C$ depending only on $n$, $N$, $m$, $s$ and $p$ such that
\[
\Vert a_\eps(x,D)\varphi\Vert_{H^{s,p}}\le C\sup_{|\alpha+\beta|\le N}| a_\eps|^{(m)}_{\alpha,\beta}\,\Vert\varphi\Vert_{H^{s+m,p}},
\]
for all $\varphi\in \S(\R^n)$ and $\eps\in(0,1]$.
\end{corollary}
\begin{proof}
Apply Theorem \ref{theo_Lp_wong} to the pseudodifferential operator with symbol $\lara{\xi}^s\sharp a_\eps(x,\xi)\sharp\lara{\xi}^{-s-m}$.
\end{proof} 

\bibliographystyle{abbrv}
\bibliography{claudia}

\end{document}